\documentclass{amsart}
\usepackage[english]{babel}
\usepackage[latin1]{inputenc}
\usepackage[dvips,final]{graphics}
\usepackage{amsmath,amsfonts,amssymb,amsthm,amscd,array,stmaryrd,mathrsfs}
\usepackage{pstricks}
 \usepackage[all]{xy}
 \usepackage{url}
\usepackage{textcomp}
 \usepackage{graphicx}
\usepackage{color}              
\usepackage{epsfig}
\usepackage{psfrag}
\usepackage{epstopdf}
\theoremstyle{plain}
\newtheorem{thm}{Theorem}
\newtheorem{lem}{Lemma}[subsection]
\newtheorem{cor}[lem]{Corollary}
\newtheorem{prop}[lem]{Proposition}

\theoremstyle{definition}
\newtheorem{defn}[lem]{Definition}
\newtheorem{rem}[lem]{Remark}
\newtheorem{ex}[lem]{Example}

\let\ssection=\section
\renewcommand{\section}{\setcounter{equation}{0}\ssection}

\newcommand{\Z}{\mathbb{Z}}

\newcommand{\Q}{\mathbb{Q}}

\newcommand{\A}{\mathcal{A}}

\newcommand{\Sc}{\mathcal{S}}

\newcommand{\SL}{\mathrm{SL}}

\begin{document}

\title[$\SL_2(\Z)$-tilings of the torus]{$\SL_2(\Z)$-tilings of the torus, Coxeter-Conway friezes 
and\\ Farey triangulations}

\author{Sophie Morier-Genoud, Valentin Ovsienko and Serge Tabachnikov}

\address{Sophie Morier-Genoud,
Sorbonne Universit\'es, UPMC Univ Paris 06, UMR 7586, Institut de Math\'ematiques de Jussieu- Paris Rive Gauche, Case 247, 4 place Jussieu, F-75005, Paris, France
}

\address{
Valentin Ovsienko,
CNRS,
Laboratoire de Math\'ematiques 
U.F.R. Sciences Exactes et Naturelles 
Moulin de la Housse - BP 1039 
51687 REIMS cedex 2,
France}

\address{
Serge Tabachnikov,
Pennsylvania State University,
Department of Mathematics,
University Park, PA 16802, USA
}

\email{sophie.morier-genoud@imj-prg.fr,
valentin.ovsienko@univ-reims.fr,
tabachni@math.psu.edu
}

\date{}



\begin{abstract}
The notion of $\SL_2$-tiling is a generalization of that of classical Coxeter-Conway
frieze pattern.
We classify doubly antiperiodic $\SL_2$-tilings that
contain a rectangular domain of positive integers.
Every such
$\SL_2$-tiling corresponds to a pair of frieze patterns and a unimodular $2\times2$-matrix
with positive integer coefficients.
We relate this notion to triangulated $n$-gons in the Farey graph.
\end{abstract}

\maketitle

\thispagestyle{empty}


\section{Introduction}
Frieze patterns were introduced and studied by Coxeter and Conway,
 \cite{Cox,CoCo}, in the 70's.
A frieze pattern is an infinite array of numbers, bounding by two diagonals of 1's, 
such that every four adjacent numbers $a,b,c,d$
forming a ``small'' square
satisfy the relation $ad-bc=1$ called \textit{the unimodular rule}; 
 for an example see Figure \ref{exCoCox}.
The \textit{width} of the frieze is the number of diagonals between the bounding diagonals of $1$'s.

The fundamental Conway-Coxeter theorem \cite{CoCo} offers the following classification:
{\it frieze patterns with positive integer entries of width $n-3$, 
are in one-to-one correspondence 
with triangulations of a convex $n$-gon};
 for a simple proof see \cite{Hen}.
More precisely, given a triangulated $n$-gon in the oriented plane, one constructs a frieze of width $n-3$ as follows.
The diagonal next to the diagonal of $1$'s 
is formed by the numbers of triangles incident at each vertex (taken cyclically).
\begin{figure}[hbtp]
$
 \begin{array}{llllllllllllllllllllllll}
\ddots&&&&&&\\
&1&2 &3 &1 &1 &1 \\
&&1&2&1&2&3&1\\
&&&1&1&3&5&2&1\\
&&&&1&4&7&3&2&1\\
&&&&&1&2&1&1&1&1\\
&&&&&&1&1&2&3&4&1\\
&&&&&&&1&3 &5 &7 &2 &1\\
&&&&&&&&1&2 &3 &1 &1&1\\
&&&&&&&&&1&2&1&2&3&1\\
&&&&&&&&&&&&&&&\ddots\\
\end{array}
$
\includegraphics[width=4cm]{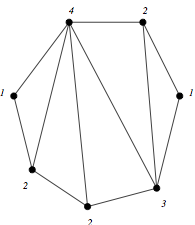}
\caption{A $7$-periodic frieze pattern and
the corresponding triangulated heptagon.}
\label{exCoCox}
\end{figure}

\noindent
This, in particular, implies that every diagonal in a frieze of width $n-3$ is $n$-periodic.
Throughout this paper, we will be considering frieze patterns with positive integer entries.

The following terminology is due to Conway and Coxeter \cite{CoCo}.
 A sequence of $n$ positive integers
 $q=(q_0,\ldots,q_{n-1})$
 is called a {\it quiddity} of order $n$,
if there exists a triangulated $n$-gon such that every~$q_i$ 
is equal to the number of incident triangles at $i$-th vertex.
For instance, the example in Figure~\ref{exCoCox} corresponds to the
following quiddities of order~$7$:
$(1,3,2,2,1,4,2),\,(3,2,2,1,4,2,1),\ldots$ (cyclic permutation).

Every quiddity of order $n$ determines a unique positive integer frieze pattern.
Two quiddities correspond to the same positive integer frieze pattern if and only if they
differ by a cyclic permutation.
According to the Conway-Coxeter theorem, positive integer frieze patterns can be enumerated
by the Catalan numbers.

 \begin{ex}
For each case $n=3,4$ and $5$, there is a unique (up to  cyclic permutation) quiddity:
$(1,1,1),\;(1,2,1,2)$ and $(1,3,1,2,2)$, respectively.

 For $n=6$, there are four different quiddities:
 $$
 (1,3,1,3,1,3),\quad(1,4,1,2,2,2),\quad(1,2,3,1,2,3),\quad(1,3,2,1,3,2)
 $$ 
 and their cyclic permutations.
 
We can also consider the ``degenerate'' case $n=2$, where
the corresponding ``degenerate'' quiddity is~$(0,0)$.

Examples of frieze patterns can be constructed using the computer program~\cite{Sch}.
\end{ex}

Among many beautiful properties of Coxeter-Conway friezes, the property of
periodicity and so-called Laurent phenomenon are particularly important.
They relate frieze patterns to the theory of cluster algebras 
developed by Fomin and Zelevinsky,
\cite{FZ1,FZ3}.

Various generalizations of  Coxeter-Conway friezes have been recently introduced and studied, 
see~\cite{CaCh, Pro, BaMa, ARS, MOT}.
One of the generalizations,
called $\SL_2$-{\it tiling}, was first considered by Assem, Reutenauer and Smith~\cite{ARS}, and further developed by Bergeron and Reutenauer~\cite{BeRe}. 
An $\SL_2$-tiling is an infinite array of numbers satisfying the above unimodular rule, 
without the condition of bounding diagonals of 1's.
Unlike the frieze patterns, $\SL_2$-tilings are not necessarily periodic.
Nevertheless,
correspondences between $\SL_2$-tilings and triangulations can be established, ~\cite{HJ,BHJ}.
\begin{figure}[hbtp]
$$
 \begin{array}{rrrrrrrrrrrrrrrrrrrrrr}
 &\vdots&\vdots&\vdots&\vdots&\vdots&\vdots&\vdots&\vdots&\vdots&\vdots&\\[4pt]
\cdots&2&5&8&11&3&-2&-5&-8&-11&-3&\cdots\\[4pt]
\cdots&7&18&29&40&11&-7&-18&-29&-40&-11&\cdots\\[4pt]
\cdots&5&13&21&29&8&-5&-13&-21&-29&-8&\cdots\\[4pt]
\cdots&3&8&13&18&5&-3&-8&-13&-18&-5&\cdots\\[4pt]
\cdots&-2&-5&-8&-11&-3&2&5&8&11&3&\cdots\\[4pt]
\cdots&-7&-18&-29&-40&-11&7&18&29&40&11&\cdots\\[4pt]
\cdots&-5&-13&-21&-29&-8&5&13&21&29&8&\cdots\\[4pt]
\cdots&-3&-8&-13&-18&-5&3&8&13&18&5&\cdots\\
 &\vdots&\vdots&\vdots&\vdots&\vdots&\vdots&\vdots&\vdots&\vdots&\vdots&
\end{array}
$$
\caption{A $(4,5)$-antiperiodic $\SL_2$-tiling with positive rectangular domain.}
\label{exTiling}
\end{figure}

The case of $(n,m)$-{\it antiperiodic, or ``toric''} $\SL_2$-tilings was suggested in~\cite{BeRe}.
In this paper, we study such tilings.

The main results of the paper are the following.

We classify doubly antiperiodic $\SL_2$-tilings that contain a 
rectangular fundamental domain of positive integers.
We show that every such $\SL_2$-tiling is generated by a pair of quiddities and a unimodular $2\times2$-matrix
with positive integer coefficients.
Although there are infinitely many such $\SL_2$-tilings, their description
is very explicit.

Following the original idea of Coxeter~\cite{Cox},
we also interpret the entries of a doubly periodic $\SL_2$-tiling
that contain a rectangular fundamental domain of positive integers in terms of 
the Farey graph of rational numbers.
Every such $\SL_2$-tiling corresponds to a triple:
an $n$-gon, an $m$-gon in the Farey graph,
and a totally positive matrix from $\SL_2(\Z)$
relating them.
We also obtain an explicit formula for the entries of the tiling.

\section{Farey graph and the Conway-Coxeter theorem}

In this section, we give an explanation of the relation between
the Coxeter frieze patterns and triangulated $n$-gons.

It was already noticed by Coxeter~\cite{Cox} that a Farey series 
(of arbitrary order $N$)
defines a frieze pattern.
Moreover, every frieze pattern corresponds to
an $n$-gon (i.e., an $n$-cycle) in the Farey graph.
A Farey $n$-gon always carries a triangulation;
we will prove that this triangulation is precisely that of Conway-Coxeter theorem.
This statement seems to be new and extend the observation illustrated in~\cite{Sch}.

\subsection{Farey graph, Farey series and Farey $n$-gons}
For two rational numbers, $v_1,v_2\in\Q$,
written as irreducible fractions $v_1=\frac{a_1}{b_1}$ and $v_2=\frac{a_2}{b_2}$,
the {\it Farey ``distance''} is defined by
$$
d(v_1,v_2):=|a_1b_2-a_2b_1|.
$$
Note that the above ``distance'' does not satisfy the triangle inequality.
Recall the definition of the {\it Farey graph}.

\begin{enumerate}
\item
The set of vertices of the Farey graph is $\Q\cup \{\infty\}$, 
with $\infty$ represented by $\frac{1}{0}$.

\item
Two vertices, $v_1,v_2$ are joined by a (non-oriented) edge $(v_1,v_2)$
whenever $d(v_1,v_2)=1$. 
\end{enumerate}

The Farey graph is often embedded into the hyperbolic half-plane,
the edges being realized as geodesics joining rational points on the ideal boundary.

The following classical properties of the Farey graph can be found in \cite{HaWr}
(the proof is  elementary).

\begin{prop}
\label{ClassProp}
(i)
Every $3$-cycle of the Farey graph is of the form
\begin{equation}
\label{TrianG}
\left\{
\frac{a_1}{b_1},\frac{a_1+a_2}{b_1+b_2},\frac{a_2}{b_2}
\right\}.
\end{equation}

(ii)
Every edge of the Farey graph belongs to a $3$-cycle.

(iii)
Edges in the Farey graph do not cross,
i.e., for a quadruple $v_1>v_2>v_3>v_4$ it is not possible to have edges $(v_1,v_3)$ and $(v_2,v_4)$.
\end{prop}

\begin{defn}
The {\it Farey series} (also called {\it Farey sequence})
of order $N$ is the sequence of irreducible fractions in $[0,1]$
whose denominators do not exceed $N$.
\end{defn}

We will write the sequences in the decreasing order; see Figure~\ref{SerFig}.
   \begin{figure}[!h]
\begin{center}
\includegraphics[width=9cm]{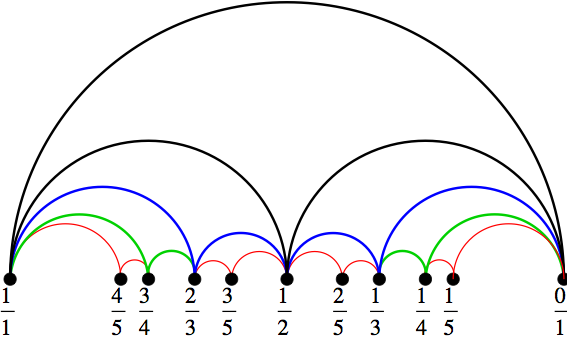}
\end{center}
\caption{The Farey series of order 5 embedded in the Farey graph}
\label{SerFig}
\end{figure}

The following fundamental property of Farey series is also proved in~\cite{HaWr}.
It shows that every Farey series is a cycle in the Farey graph.

\begin{prop}
\label{ClassPropBis}
Every two consecutive numbers in a Farey series are joined by an edge in the Farey graph.
\end{prop}

This is less elementary than Proposition~\ref{ClassProp}, so we propose here a short proof.
Our proof is different from the well-known one,
it is based on the classical Pick formula.

\begin{proof}
Consider two consecutive numbers $\frac{a}{b}>\frac{c}{d}$, in a Farey series of some order $N$. 
Suppose that $ad -bc\geq 2$. 
The quantity $A=\frac12(ad-bc)$ is the area of the Euclidean triangle spanned by the vertices
$(0,0)$, $(a,b)$, $(c,d)$.
Pick's formula states:
$$
A=I+\frac{B}{2}-1,
$$
where $I$ is the number of integer points in the interior of the triangle, 
and $B$ the number of integer points on the border.
By assumption,~$A\geq 1$, and therefore $I+\frac{B}{2}\geq2$.
It follows that there exists a point $(x,y)$,
which is either inside the triangle, or on the segment between $(a,b)$ and $(c,d)$
(since the fractions $\frac{a}{b}$ and $\frac{c}{d}$ are irreducible). 
One then has: 
$$
y\leq \max(b,d)\leq N
\qquad\hbox{and}\qquad
\frac{a}{b}>\frac{x}{y}>\frac{c}{d}.
$$
This contradicts the assumption that $\frac{a}{b}$ and $\frac{c}{d}$ are
consecutive numbers in the Farey series.
\begin{figure}[!h]
\begin{center}
\includegraphics[width=5cm]{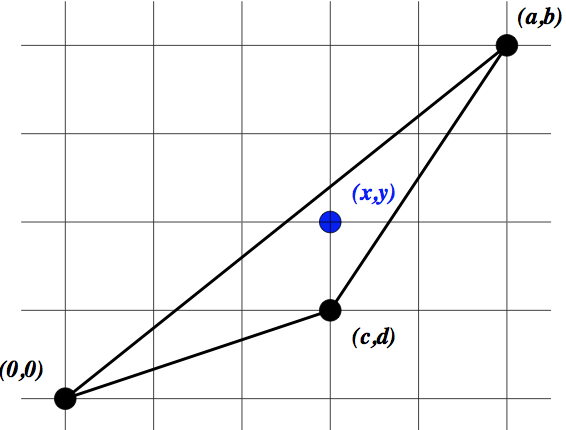}
\end{center}
\caption{The case of interior point}
\label{PickFig}
\end{figure}
\end{proof}

Proposition~\ref{ClassPropBis} is used three times to prove the following.

\begin{cor}
\label{TriangCor}
Every Farey series forms a triangulated polygon in the Farey graph.
\end{cor}

\begin{proof}
We prove this statement by induction on $N$ (the order of Farey series).
Assume that the series of order~$N-1$ is triangulated.
The series of order $N$ is obtained from that of order $N-1$ by adding
points of the form $\frac{k}{N}$.

First, we observe that two points, $\frac{k_1}{N}$
and $\frac{k_2}{N}$ cannot be consecutive.
Indeed, $d(\frac{k_1}{N},\frac{k_2}{N})\not=1$: that would contradict Proposition~\ref{ClassPropBis};
therefore, every new point $\frac{k}{N}$ appears between two ``old'' points:
\begin{equation}
\label{Tri}
\frac{p_1}{q_1}>\frac{k}{N}>\frac{p_2}{q_2}.
\end{equation}
Second, by Proposition~\ref{ClassPropBis}, $\frac{k}{N}$ is joined by edges with 
$\frac{p_1}{q_1}$ and $\frac{p_2}{q_2}$.
Third, $\frac{p_1}{q_1}$ and $\frac{p_2}{q_2}$
are joined by an edge, according to Proposition~\ref{ClassPropBis}
applied to the series of order $N-1$.
We conclude that~(\ref{Tri}) is a triangle.
\end{proof}

We will be interested in $n$-cycles (or ``$n$-gons'') in the Farey graph
that are more general than Farey series.

\begin{defn}
\begin{enumerate}
\item
An $n$-{\it gon} in the Farey graph, or a \textit{Farey $n$-gon} is
a decreasing sequence of rationals $(v_0,\ldots,v_{n-1})$:
$$
\infty\geq{}v_0>v_1>\ldots>{}v_{n-1}\geq0,
$$ 
such that every pair of consecutive numbers $v_i,v_{i+1}$,
as well as $v_{n-1},v_0$, are joined by an edge. 

\item
The $n$-gon is called  \textit{normalized} if $v_0=\infty$ and $v_{n-1}=0$.
\end{enumerate}
\end{defn}

Since every $n$-gon can be embedded in a Farey series, Corollary \ref{TriangCor} implies the following.

\begin{cor}
Every Farey $n$-gon is triangulated.
\end{cor}
\noindent
We thus can speak of the {\it quiddity of a Farey $n$-gon}.

\begin{proof}
A Farey $n$-gon is obtained from a Farey series which is a triangulated
polygon, by cutting along diagonals of the triangulation.
\end{proof}

We define the notion of {\it cyclic equivalence}
of Farey $n$-gons.
Given an $n$-gon $(v_0,\ldots,v_{n-1})$, consider the $n$-cycle
$(v_1,\ldots,v_{n-1},v_0)$, and renormalize it using the $\SL_2(\Z)$-action so that
$v_1=\infty$ and $v_0=0$.
The obtained $n$-gon is called cyclically equivalent to the given one.
For an example, see Figure~\ref{ExFig}.

   \begin{figure}[!h]
\begin{center}
\includegraphics[width=14cm]{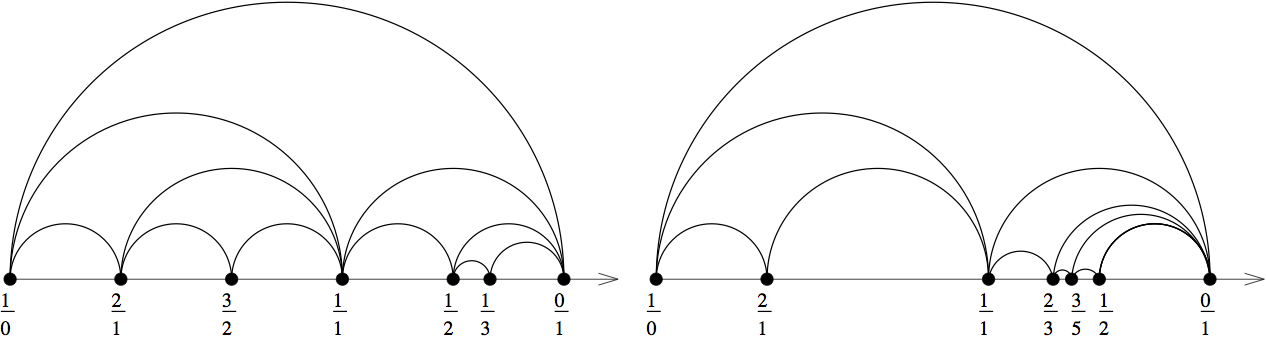}
\end{center}
\caption{Two cyclically equivalent normalized heptagons in the Farey graph corresponding to the frieze of 
Figure~\ref{exCoCox}}
\label{ExFig}
\end{figure}

\subsection{Farey $n$-gons and Coxeter-Conway friezes}
Proposition~\ref{ClassPropBis} leads to the following observation due to Coxeter~\cite{Cox}:
{\it every Farey series gives rise to a Coxeter-Conway frieze pattern of positive integers.}
Along the same lines, we have the following strengthened statement.

\begin{prop}
The Coxeter-Conway frieze patterns of positive integers of width $n-3$ are in one-to-one correspondence with the
normalized Farey $n$-gons, up to cyclic equivalence.
\end{prop}

\begin{proof}
The correspondence is given by considering the ratios of 
two consecutive rows of the frieze patterns.
The sequence
$$
v_0=\frac{1}{0},\quad
v_1=\frac{a_1}{1},\quad\ldots,\quad
v_i=\frac{a_{i}}{b_{i}},\quad\ldots,\quad
v_{n-2}=\frac{1}{b_{n-2}},\quad
v_{n-1}=\frac{0}{1}
$$
corresponds to the frieze determined by the rows
$$
\begin{array}{ccccccc}
1&a_1&a_2&\cdots&a_{n-3}&1&0\\[4pt]
0&1&b_2&&\cdots&b_{n-2}&1\\
\end{array}
$$
and {\it vice versa}.
\end{proof}

The Conway-Coxeter theorem mentioned in the introduction
provides a relation between frieze patterns and triangulations.
The following result somewhat ``demystifies'' this relation
and provides an alternative proof of the Conway-Coxeter theorem.

\begin{thm}
\label{fareyquid}
The quiddity of a Farey $n$-gon coincides with the quiddity of the 
corresponding Coxeter-Conway frieze pattern.
\end{thm}

\begin{proof}
Consider a frieze pattern, and
denote by $c_{i,j}$ its entries:
$$
\begin{array}{lllllllll}
0&1&c_{1,1}&c_{1,2}&\cdots&c_{1,n-3}&1&0\\[6pt]
&0&1&c_{2,2}&&\cdots&c_{2,n-2}&1\\[4pt]
&&\ddots&\ddots&&&&\ddots
\end{array}
$$
where 
$$
\left\{
\begin{array}{rl}
c_{i,j}=1, & i-j=1\; \hbox{or}\; 3-n,\\[4pt]
c_{i,j}=0, & i-j=2\; \hbox{or}\; 2-n.
\end{array}
\right.
$$
The quiddity of the frieze pattern reads in the $n$-periodic line $(c_{i,i})$.

Clearly, two consecutive rows determine the rest of the frieze;
the following formula was proved in~\cite{Cox}, formula~(5.6):
$$
c_{i,j}= c_{1,i-2}c_{2,j}-c_{1,j}c_{2,i-2}.
$$
In particular, we have:
\begin{eqnarray}
\label{cij}
c_{i,i}= c_{1,i-2}c_{2,i}-c_{1,i}c_{2,i-2}.
\end{eqnarray}

The corresponding  Farey $n$-gon has the following vertices 
$$
v_0=\frac{1}{0},\quad
v_1=\frac{c_{1,1}}{1},\quad\ldots\quad
v_i=\frac{c_{1,i}}{c_{2,i}},\quad\ldots\quad
v_{n-2}=\frac{1}{c_{2,n-2}},\quad
v_{n-1}=\frac{0}{1}.
$$
Therefore, the expression~(\ref{cij}) reads: $c_{i,i}=d(v_{i-2},v_i)$.
It remains to calculate the Farey distance between pairs of vertices
$v_{i-2}$ and $v_i$ in a Farey $n$-gon.

\begin{lem}
\label{DistLem}
Given a (triangulated) Farey $n$-gon
$$
v_0=\frac{1}{0},\quad
v_1=\frac{a_1}{1},\quad\ldots,\quad
v_i=\frac{a_{i}}{b_{i}},\quad\ldots,\quad
v_{n-2}=\frac{1}{b_{n-2}},\quad
v_{n-1}=\frac{0}{1},
$$
the Farey distance $d(v_{i-1},v_{i+1})$ coincides with the number of
triangles incident at $v_i$.
\end{lem}

\begin{proof}
Among all the vertices of the $n$-gon $(v_i)$, let us select those connected to $v_i$
by edges of the Farey graph. 
Denote by $\{v_{i_1},\ldots,v_{i_k}\}$, resp. $\{v_{i_{k+1}},\ldots,v_{i_{k+\ell}}\}$
the vertices at the left, resp. right, of $v_i$, so that 
$$
v_{i_1}> \ldots >v_{i_k}>v_i>v_{i_{k+1}}>\ldots>v_{i_{k+\ell}},
$$
(note that $v_{i_k}=v_{i-1}$ and $v_{i_{k+1}}=v_{i+1}$).
The number of triangles incident at $v_i$ is then equal to~$k+\ell-1$.

Two consecutive selected vertices, $v_{i_j}$ and $v_{i_{j+1}}$ are
connected by an edge.
Indeed, this follows from the fact that every Farey polygon is triangulated.
Therefore, the vertices $(v_{i_j}, v_{i_{j+1}}, v_i)$ form a triangle (a $3$-cycle) in the Farey graph.
Using Eq.~(\ref{TrianG}), we obtain by induction:
$$
v_{i-1}(=v_{i_k})=\dfrac{a_{i_1}+(k-1)a_i}{b_{i_1}+(k-1)b_i}, \qquad 
v_{i+1}(=v_{i_{k+1}})=\dfrac{a_{i_{k+\ell}}+(\ell-1)a_i}{b_{i_{k+\ell}}+(\ell-1)b_i}.
$$
We have:
$$
d(v_{i-1},v_{i+1})=
a_{i_1}b_{i_{k+\ell}}-b_{i_1}a_{i_{k+\ell}}+(k-1)(a_ib_{i_{k+\ell}}-b_ia_{i_{k+\ell}})+
(\ell-1)(a_{i_1}b_{i}-b_{i_{1}}a_{i_{}}).
$$

By assumption, $v_i$ is joined by edges with $v_{i_{1}}$ and $v_{i_{k+\ell}}$, hence
$a_ib_{i_{k+\ell}}-b_ia_{i_{k+\ell}}=1$, and $a_{i_1}b_{i}-b_{i_{1}}a_{i_{}}=1$.
Furthermore, $(v_{i_1}, v_i, v_{i_{k+\ell}})$ is also a triangle, therefore
$a_{i_1}b_{i_{k+\ell}}-b_{i_1}a_{i_{k+\ell}}=1$.
We have finally:
\begin{equation}
\label{DisTEq}
d(v_{i-1},v_{i+1})=k+\ell-1.
\end{equation}
Hence the lemma.
\end{proof}

Theorem~\ref{fareyquid} is proved.
\end{proof}

\subsection{Entries of the frieze pattern}
Coxeter's formula~(5.6) in \cite{Cox} for the entries of the frieze pattern
translates into our language as the following general expression:
\begin{equation}
\label{Frizaij}
c_{i,j}=d(v_{i-2},v_j),
\end{equation}
where, as above, $(v_i)$ is the Farey $n$-gon corresponding to the frieze pattern.

\section{$\SL_2$-tilings}

In this section, we introduce the main notions
studied in this paper.

\subsection{Tame $\SL_2$-tilings}
Let us first recall the notion of $\SL_2$-tiling introduced in \cite{BeRe}. 

 \begin{enumerate}
\item An $\SL_2$-tiling, is an infinite matrix $\A=(a_{i,j})_{(i,j)\in\Z\times \Z}$, 
such that every adjacent $2\times 2$-minor equals~$1$:
$$
\begin{vmatrix}
a_{i,j}&a_{i,j+1}\\
a_{i+1,j}&a_{i+1,j+1}
\end{vmatrix}=1,
$$
for all $(i,j)\in\Z\times \Z$.
\item 
The tiling is called \textit{tame} if  every adjacent $3\times 3$-minor equals~$0$:
$$
\begin{vmatrix}
a_{i,j}&a_{i,j+1}&a_{i,j+2}\\
a_{i+1,j}&a_{i+1,j+1}&a_{i+1,j+2}\\
a_{i+2,j}&a_{i+2,j+1}&a_{i+2,j+2}\\
\end{vmatrix}=0,
$$
for all $(i,j)\in\Z\times \Z$.
 \end{enumerate}
 
 Let us stress on the fact that a {\it generic} $\SL_2$-tiling is tame.
 
\subsection{Antiperiodicity}
The following condition was also suggested in~\cite{BeRe}.
 
 An $\SL_2$-tiling is called $(n,m)$-{\it antiperiodic} if every row is 
 $n$-antiperiodic, and every column is $m$-antiperiodic:
$$
\begin{array}{rcl}
 a_{i,j+n}&=&-a_{i,j}\;,\\[4pt] 
a_{i+m,j}&=&-a_{i,j}\;,
\end{array}
$$
for all $(i,j)\in\Z\times \Z$.

The following relation between $(n,m)$-antiperiodic $\SL_2$-tilings
and the classical Coxeter-Conway frieze patterns
shows that the antiperiodicity condition for the $\SL_2$-tilings 
is natural and interesting.

\subsection{Frieze patterns and $(n,n)$-antiperiodic $\SL_2$-tilings}
As explained in~\cite{BeRe},
every Coxeter-Conway frieze pattern of width $n-3$ can be extended to a tame
$(n,n)$-antiperiodic $\SL_2$-tiling,
in a unique way.

The construction is as follows. 
One adds two diagonals of $0$'s next to the diagonals of $1$'s, and then continues by antiperiodicity.

 \begin{ex}
The frieze pattern in Figure \ref{exCoCox}
corresponds to the following $(7,7)$-antiperiodic tame $\SL_2$-tiling.
$$
 \begin{array}{rrrrrrrrrrrrrrrrrrrrrr}
 &\vdots&\vdots&\vdots&\vdots&\vdots&\vdots&\vdots&\vdots&\vdots&\vdots&\vdots&\vdots&\\
\cdots&1&2 &3 &1 &1 &1 &0 &-1&-2&-3&-1&-1 &\cdots\\
\cdots& 0&1&2&1&2&3&1&0&-1&-2&-1&-2&\cdots\\
\cdots& -1&0&1&1&3&5&2&1&0&-1&-1&-3&\cdots\\
\cdots&-2&-1& 0&1&4&7&3&2&1&0&-1&-4&\cdots\\
\cdots&-1&-1&-1&0&1&2&1&1&1&1&0&-1&\cdots\\
\cdots&-2&-3&-4&-1&0&1&1&2&3&4&1&0&\cdots\\
\cdots&-3&-5&-7&-2&-1 &0&1&3 &5 &7 &2 &1&\cdots\\
\cdots&-1&-2&-3&-1&-1 &-1&0&1&2 &3 &1 &1 & \cdots\\
\cdots&0&-1&-2&-1&-2&-1 &-1&  0&1&2&1&2&\cdots\\
 &\vdots&\vdots&\vdots&\vdots&\vdots&\vdots&\vdots&\vdots&\vdots&\vdots&\vdots&\vdots&\\
\end{array}
$$
\end{ex}

For the details
of the above construction and the ``antiperiodic nature'' of Conway-Coxeter's friezes;
see~\cite{BeRe,MOST}.

\subsection{Positive rectangular domain}

In this paper, we are considering $(n,m)$-antiperiodic $\SL_2$-tilings that
contain an $m\times n$-rectangular domain of positive integers.

More precisely, we are interested in $\SL_2$-tilings of the following form:
\begin{equation}
\label{FormEq}
 \begin{array}{c|ccc|ccc|ccc|c}
 &&\vdots&&&\vdots&&&\vdots&&\\
 \hline
  &&&&&&&&&&\\
  \cdots &&P&&&-P&&&P&& \cdots\\
    &&&&&&&&&&\\
   \hline
  &&&&&&&&&&\\
  \cdots &&-P&&&P&&&-P&& \cdots\\
      &&&&&&&&&&\\
   \hline
 &&\vdots&&&\vdots&&&\vdots&&\\
\end{array}
\end{equation}
where $P$ is an $m\times n$-matrix with entries in $\Z_{>0}$.
An example of such an $\SL_2$-tilling is presented in Figure~\ref{exTiling}.

The following property is important for us.

\begin{prop}
\label{TLem}
An $(n,m)$-antiperiodic $\SL_2$-tiling that
contains a positive $m\times n$-rectangular domain is tame.
\end{prop}

\begin{proof}
This is a consequence of the Jacobi identity or Dodgson formula on determinants:
$$
\begin{vmatrix}
\bullet & \bullet & \bullet \\
\bullet & \bullet & \bullet \\
\bullet & \bullet & \bullet \\
\end{vmatrix}
\begin{vmatrix}
\circ  & \circ  & \circ  \\
\circ  & \bullet & \circ  \\
\circ  & \circ  & \circ  \\
\end{vmatrix}
=
\begin{vmatrix}
\bullet & \bullet &  \circ  \\
\bullet & \bullet &  \circ  \\
 \circ  &  \circ  &  \circ  \\
\end{vmatrix}
\begin{vmatrix}
 \circ  &  \circ  &  \circ  \\
 \circ  & \bullet & \bullet \\
 \circ  & \bullet & \bullet \\
\end{vmatrix}
-
\begin{vmatrix}
 \circ  &  \circ  &  \circ  \\
\bullet & \bullet &  \circ  \\
\bullet & \bullet &  \circ  \\
\end{vmatrix}
\begin{vmatrix}
 \circ  & \bullet & \bullet \\
 \circ  & \bullet & \bullet \\
 \circ  &  \circ  &  \circ  \\
\end{vmatrix}
$$
where the white dots represent deleted entries, and the black dots initial entries.

Since the values are non zero and the $2\times2$-minors all equal to $1$, 
the above identity implies that all the  $3\times 3$-minors vanish.
\end{proof}

\section{The main theorem}

In this section, we formulate our main result.
The proof will be given in Section~\ref{PSec}.

\subsection{Classification}

It turns out that every $\SL_2$-tiling corresponds to a pair of frieze patterns and
a positive integer $2\times2$-matrix $M$ satisfying some conditions.

\begin{thm}
\label{thethm} 
The set of $(n,m)$-antiperiodic $\SL_2$-tilings containing a fundamental rectangular domain of positive integers 
is in a one-to-one correspondence with the set of triples $(q,q', M)$, where
$$
q=(q_0,\ldots,q_{n-1}),
\qquad
q'=(q'_0,\ldots,q'_{m-1})
$$ 
are quiddities of order $n$ and $m$, respectively, and where
$M=\begin{pmatrix}
a&b\\
c&d
\end{pmatrix}$ is a unimodular $2\times2$-matrix with positive integer coefficients, such that the inequalities
\begin{equation}
\label{StrangeEq}
q_0<\frac{b}{a},
\qquad
q'_0<\frac{c}{a}
\end{equation}
are satisfied.
\end{thm}

\begin{rem}
It is important to notice that inequalities (\ref{StrangeEq}) also imply
\begin{equation}
\label{StrangeEqBis}
q_0<\frac{d}{c},
\qquad
q'_0<\frac{d}{b}.
\end{equation}
Indeed, the unimodular condition $ad-bc=1$ and
the assumption that $a,b,c,d$ are positive integers imply
that $\frac{b}{a}<\frac{d}{c}$ and $\frac{c}{a}<\frac{d}{b}$.
\end{rem}

\begin{cor}
For every pair of quiddities $q,q'$,
there exist infinitely many $(n,m)$-antiperiodic $\SL_2$-tilings containing a
fundamental rectangular domain of positive integers.
\end{cor}

\begin{proof}
Given arbitrary pair of quiddities $q$ and $q'$,
the matrices:
$$
\begin{pmatrix}
1&b\\[4pt]
c&bc+1
\end{pmatrix}
$$
satisfy~(\ref{StrangeEq}) for sufficiently large $b,c$.
\end{proof}

\subsection{The semigroup $\Sc$}
Consider the set of $2\times2$-matrices with positive integral entries
satisfying the following conditions of positivity:
\begin{equation}
\label{SGEq}
 \Sc=
\left\{ \begin{pmatrix}
a&b\\
c&d
\end{pmatrix}
\in \SL_2(\Z)
\right.
\left|
\begin{array}{l}
0<a<b<d,\\[2pt]
0<a<c<d
\end{array}
\right\}.
\end{equation}
Note that the inequalities $b<d$ and $c<d$ are included for the sake of completeness.
These inequalities actually follow from $a<b,\;a<c$ together with $ad-bc=1$ and
the assumption that $a,b,c,d$ are positive.

We have the following property.

\begin{prop}
The set $\Sc\subset\SL_2(\Z)$ is a {\it semigroup}, i.e., it is stable by multiplication.
\end{prop}

\begin{proof}
Straightforward.
\end{proof}

The semigroup~$\Sc$ naturally appears in our context.
Indeed,
if $n,m\geq3$, then the inequalities~(\ref{StrangeEq}) imply $M\in\Sc$.
Moreover every quiddity $q$ contains a unit entry, so that after a cyclic permutation
of any quiddity one can obtain $q_0=1$.
The inequalities~(\ref{StrangeEq}) then coincide with the conditions~(\ref{SGEq}).


\subsection{Examples}
Let us give two simple examples of $\SL_2$-tilings.

\begin{ex}
There is a one-to-one correspondence between
$(3,3)$-antiperiodic $\SL_2$-tilings containing a
fundamental domain of positive integers and elements of the semigroup $\Sc$.
Indeed, the only quiddity of order $3$ is $q=(1,1,1)$.
To every matrix~(\ref{SGEq}) there corresponds the following $\SL_2$-tiling:
$$
\begin{array}{ccccc}
&\vdots&\vdots&\vdots&\\
\cdots&a&b&b-a&\cdots\\[8pt]
\cdots&c&d&d-c&\cdots\\[8pt]
\cdots&c-a&d-b&d-b-c+a&\cdots\\
&\vdots&\vdots&\vdots&\\
\end{array}
$$
It is a good exercise to check that the positivity condition
$d-b-c+a>0$ follows from~(\ref{SGEq}) together with $ad-bc=1$.
\end{ex}

\begin{ex}
In the case $n=2$ or $m=2$,  the conditions~(\ref{StrangeEq}) become trivial.

Consider also the simplest (degenerate) case of $(2,2)$-antiperiodic $\SL_2$-tilings.
A $(2,2)$-antiperiodic $\SL_2$-tiling containing a
fundamental domain of positive integers is of the form:
$$
\begin{array}{cccccc}
&\vdots&\vdots&\vdots&\vdots&\\[4pt]
\cdots&a&b&-a&-b&\cdots\\[4pt]
\cdots&c&d&-c&-d&\cdots\\
&\vdots&\vdots&\vdots&\vdots&\\
\end{array}
$$
where $\begin{pmatrix}
a&b\\
c&d
\end{pmatrix}$ is an arbitrary unimodular matrix with positive integer coefficients.
Note that this case corresponds to the ``degenerate quiddity'' of order $2$, namely $q=(0,0)$.
\end{ex}

\section{Frieze patterns and linear recurrence equations}

We will recall here a remarkable and well-known property of Coxeter-Conway frieze patterns.
It concerns a  relation of frieze patterns and linear recurrence equations.
The statement presented in this subsection was implicitly obtained in~\cite{CoCo};
 for details see~\cite{MOST}.
We recall this statement without proof.

\subsection{Discrete non-oscillating Hill equations}

\begin{defn}
Let $(c_i)_{i\in \Z}$ be an arbitrary $n$-periodic sequence of numbers.
\begin{enumerate}
\item[(a)]
A linear difference equation
\begin{equation}
\label{Sro}
V_{i+1}=c_iV_i-V_{i-1},
\end{equation}
where the sequence $(c_i)$ is given (the coefficients)
and where $(V_i)$ is unknown (the solution),
is called a discrete Hill, or Sturm-Liouville, or
one-dimensional Schr\"odinger equation.
\item[(b)]
The equation~(\ref{Sro}) is called {\it non-oscillating} if every solution $(V_i)$
is antiperiodic:
$$
V_{i+n}=-V_i,
$$
for all $i$, and  has exactly one sign change in any sequence $(V_i, V_{i+1}, \ldots V_{i+n})$.
\end{enumerate}
\end{defn}

In other words,
every solution of a non-oscillating equation must have non-negative intervals of length $n$,
that is, $n$ consecutive non-negative values: 
$
(V_k,\ldots,V_{k+n-1}).
$
 
Moreover, for  {\it generic} solution of~(\ref{Sro}), 
all the elements $V_j$ of a non-negative interval are {\it strictly positive}.
Zero values can only occur at the endpoints: $V_k=0$, or $V_{k+n-1}=0$.

Note also that the coefficients in a non-oscillating equation are necessarily positive.

\subsection{Frieze patterns and difference equations}
The relation between the equations~(\ref{Sro}) and Coxeter-Conway frieze patterns
is as follows.

\begin{prop}
\label{Known}
Given an equation \eqref{Sro} with integer coefficients, it is a non-oscillating equation if and only if 
the coefficients $(c_0,c_1,\ldots, c_{n-1})$ form a quiddity.
\end{prop}

\begin{proof}
This is an immediate consequence of properties  established by Coxeter and Conway. 
Indeed, it was proved in \cite{Cox} (see also \cite{CoCo} property (17))
that the entries in any row of the pattern (extended by antiperiodicity) form a solution of an equation \eqref{Sro}, where
the coefficients $c_i$ are given by the sequence on the first non-trivial diagonal.
Thus, from an non-oscillating equation one can write down a frieze, and vice versa.
$$
\begin{array}{ccccccccccccc}
\ddots&\ddots&&\ddots&\ddots&\ddots\\
&1&c_0&\cdots&1&0&-1&\cdots\\[6pt]
&&1&c_1&\cdots&1&0&-1&\cdots\\[4pt]
&&&1&c_2&\cdots&1&0&-1&\cdots\\[2pt]
&&&&\ddots&\ddots&&\ddots&\ddots&\ddots\end{array}
$$
Finally, the integer condition establish the correspondence with quiddities.
\end{proof}

Of course, for an arbitrary non-oscillating equation~(\ref{Sro}),
the corresponding frieze pattern does not necessarily have integer entries.
In~\cite{MOST}, the space of frieze patterns and the space of non-oscillating equation~(\ref{Sro})
are identified in a more general setting.

\begin{ex}
(a)
The simplest quiddity $q=(1,1,1)$
corresponds to the non-oscillating equation with all $c_i=1$.
Every solution of this equation is $3$-antiperiodic and can be obtained as a linear combination
of the following two solutions:
$$
(V^{(1)}_i)=(\ldots,0,1,1,0,-1,-1,\ldots),
\qquad
(V^{(2)}_i)=(\ldots,1,1,0,-1,-1,0\ldots).
$$
This corresponds to a degenerate frieze of Coxeter-Conway of width 0.
(b)
The frieze from Figure~\ref{exCoCox} corresponds
to the non-oscillating equation with $7$-antiperiodic solutions that
are linear combinations of the following two:
$$
(V^{(1)}_i)=(\ldots,1,2,3,1,1,1,0,\ldots),
\qquad
(V^{(2)}_i)=(\ldots,0,1,2,1,2,3,1,\ldots).
$$
The above two solutions are exactly the first two rows of the frieze in Figure~\ref{exCoCox}.
One can of course choose different rows for a basis.

Note that, in the both cases, the basis solutions 
$(V^{(1)}_i),(V^{(2)}_i)$ are not generic since they contain zeros.
\end{ex}

\section{Proof of Theorem \ref{thethm}}\label{PSec}

\subsection{The construction}

Given a triple $(q,q', M)$ as in Theorem \ref{thethm},
we will construct an $\SL_2$-tiling satisfying the above conditions.
Define
$T=(a_{i,j})$ using the following recurrence relations:
\begin{equation}
\label{RecEq}
\begin{array}{rcl}
a_{i,j+1}&:=&q_ja_{i,j}-a_{i,j-1},\\[4pt]
a_{i+1,j}&:=&q'_ia_{i,j}-a_{i-1,j},
\end{array}
\end{equation}
for all $i,j\in\Z$, where the quiddities are periodically extended, i.e $q_i=q_{i+n}, q'_i=q'_{i+m}$,
and taking the initial conditions
\begin{equation}
\label{Init}
\begin{pmatrix}
a_{0,0}&a_{0,1}\\[4pt]
a_{1,0}&a_{1,1}
\end{pmatrix}
:=
 \begin{pmatrix}
a&b\\[2pt]
c&d
\end{pmatrix}.
\end{equation}
It is very easy to check that the tiling $T$ is well-defined,
i.e., the two recurrences commute and
the calculations along the rows and columns give the same result.
We show that the defined tiling $T$ 
contains a fundamental rectangular domain of positive integers.

By Proposition~\ref{Known}, the defined tiling $T$ is $(n,m)$-antiperiodic.
Consider the following $m\times n$-subarray of $T$
\begin{equation}
\label{Ptile}
P=
\begin{pmatrix}
a_{0,0}&a_{0,1}&\cdots&a_{0,n-1}\\[4pt]
a_{1,0}&a_{1,1}&\cdots&a_{1,n-1}\\[4pt]
\cdots&&&\\
a_{m-1,0}&a_{m-1,1}&\cdots&a_{m-1,n-1}
\end{pmatrix}.
\end{equation}

The main step of the proof of Theorem~\ref{thethm} is the following lemma.

\begin{lem}
\label{posP}
The entries of $P$ are positive integers.
\end{lem}

\begin{proof}
It turns out that thanks to Proposition~\ref{Known} we will only need to
perform ``local'' calculation of the elements neighboring to the initial ones:  
$$
\begin{array}{c|cc}
a_{-1,-1}&a_{-1,0}&a_{-1,1}\\[4pt]
\hline
a_{0,-1}&a&b\\[4pt]
a_{1,-1}&c&d
\end{array}
$$

The conditions~(\ref{StrangeEq}) imply:
$a_{0,-1}<0$ and $a_{-1,0}<0.$
Indeed, from (\ref{RecEq}) and (\ref{Init}), one has
$$
a_{0,-1}=q_0a-b,
\qquad
a_{-1,0}=q'_0a-c.
$$
Since the rows and the columns of $P$ are solutions of
non-oscillating equations, and $a$ is positive, this implies that
all the values of the first row and the first column of $P$ are positive.

Furthermore, again from the recurrence (\ref{RecEq}), one has
$$
a_{-1,-1}=q_0q'_0a-q_0c-q'_0b+d.
$$
The condition (\ref{StrangeEq}) then implies $a_{-1,-1}>0$.
Indeed, one establishes
$$
0<q_0=aq_0(d-q'_0b)-bq_0(c-q'_0a)<b(d-q'_0b)-bq_0(c-q'_0a)=b(q_0q'_0a-q_0c-q'_0b+d).
$$

Proposition~\ref{Known} then guarantees that
$$
\begin{array}{lcr}
a_{0,-1}<0,
&\ldots,&
a_{m-1,-1}<0,\\[4pt]
a_{-1,0}<0,
&\ldots,&
a_{-1,n-1}<0,
\end{array}
$$
and applying again Proposition~\ref{Known}, we deduce that all the entries in $P$ are positive.
\end{proof}

\subsection{From tilings to triples}
Conversely,
consider a $(n,m)$-periodic $\SL_2$-tiling  $T=(a_{i,j})_{(i,j)\in\Z\times \Z}$ 
such that the $m\times n$-subarray $P$ given by~(\ref{Ptile})
consists in positive integers.
We claim that $T$ can be obtained by the above construction.

\begin{lem}\label{ratios}
The ratios of the first two rows of $P$ form a decreasing sequence:
$$
\frac{a_{0,0}}{a_{1,0}} > \frac{a_{0,1}}{a_{1,1}} > \ldots > \frac{a_{0,n-1}}{a_{1,n-1}}, 
$$
and similarly for the ratios of the first two columns of $P$:
$$
\frac{a_{0,1}}{a_{0,0}} > \frac{a_{1,1}}{a_{1,0}} > \ldots > \frac{a_{m-1,1}}{a_{m-1,0}}.
$$

\end{lem}

\begin{proof}
This follows from the unimodular conditions $a_{0,j}a_{1,j+1}-a_{0,j+1}a_{1,j}=1$ 
and the assumption that all the entries of $P$ are positive.
\end{proof}

\begin{lem}
\label{DefRec}
The entries of $T$ satisfy the recurrence relations~(\ref{RecEq})
where $q=(q_j)$ and $q'=(q'_i)$ are $n$-periodic and $m$-periodic
sequences of positive integers, respectively.
\end{lem}

\begin{proof}
Given $(i,j)$, there is a linear relation
$$
\begin{pmatrix}
a_{i,j+1}\\
a_{i+1,j+1}
\end{pmatrix}
=\lambda_{i,j}
\begin{pmatrix}
a_{i,j}\\
a_{i+1,j}
\end{pmatrix}+
\mu_{i,j}
\begin{pmatrix}
a_{i,j-1}\\
a_{i+1,j-1}
\end{pmatrix}.
$$
Using the  $\SL_2$ conditions one immediately obtains the values
$$
\lambda_{i,j}= a_{i,j-1}a_{i+1,j+1}-a_{i,j+1}a_{i+1,j-1}, \qquad \mu_{i,j}=-1.
$$
From Lemma \ref{ratios}, one has $\lambda_{i,j}>0$.
Furthermore, it readily follows from the tameness property
(see Proposition~\ref{TLem}) 
that $\lambda_{i,j}$ actually does not depend on $i$, so we use the notation $q_j:=\lambda_{i,j}$.

The arguments for the rows are similar.
\end{proof}

\begin{lem}
\label{DefRec}
The above sequences $(q_0,\ldots,q_{m-1})$ and $(q'_0,\ldots,q_{n-1})$ are quiddities.\end{lem}
\begin{proof}
The rows, resp. columns, of $T$ are antiperiodic solutions of an equation \eqref{Sro}
with $c_i=c_{i+n}=q_i$, resp. $c_i=c_{i+m}=q'_i$. 
It follows from Proposition~\ref{Known} that the coefficients are quiddities.
\end{proof}
\begin{lem}
The $2\times 2$ left upper block of $P$, 
satisfies
$$
\begin{array}{rcl}
q_0\,a_{0,0}&<&a_{0,1},\\[4pt]
q'_0\,a_{0,0}&<&a_{1,0}.
\end{array}
$$
\end{lem}

\begin{proof}
By antiperiodicity, $a_{0,-1}<0$.
One has from~(\ref{RecEq}): $a_{0,1}=q_0\,a_{0,0}-a_{0,-1}$, and similarly for~$q'_0$.
Hence the result.
\end{proof}

In other words, the elements of the matrix 
$$
\begin{pmatrix}
a_{0,0}&a_{0,1}\\[4pt]
a_{1,0}&a_{1,1}
\end{pmatrix}
=:
 \begin{pmatrix}
a&b\\[2pt]
c&d
\end{pmatrix}
$$
satisfy~(\ref{StrangeEq}).

Theorem~\ref{thethm} is proved.

\section{$\SL_2$-tilings and the Farey graph}

In this section, we give an interpretation of the entries $a_{i,j}$
of a doubly periodic $\SL_2$-tiling.
We follow the idea of Coxeter~\cite{Cox}
and consider $n$-gons in the classical Farey graph.

\subsection{The distance between two $n$-gons}

Consider a doubly periodic $\SL_2$-tiling $T=(a_{i,j})$ and 
the corresponding triple $(q,q', M)$ (see Theorem~\ref{thethm}). 
Our next goal is to give an explicit expression for the numbers $a_{i,j}$
similar to~(\ref{Frizaij}).

From the triple $(q,q', M)$
 we construct the unique $n$-gon $(v_0,v_1,\ldots,v_{n-1})$ 
 and the unique $m$-gon $(v'_0,v'_1,\ldots,v'_{m-1})$
 with the ``initial'' conditions:
$$
\textstyle
\left(v_0,v_1\right):=\left(\frac{a}{c},\,\frac{b}{d}\right),
\qquad
\left(v'_0,v'_{m-1}\right):=\left(\frac{1}{0},\,\frac{0}{1}\right),
$$
and with the quiddities $(q_0,\ldots,q_{n-1})$  and
$(q'_1,\ldots,q'_m)$, respectively.
Notice that the quiddity $q'$ is shifted cyclically.

 \begin{thm}
 \label{theprop}
The entries of the $\SL_2$-tiling $T=(a_{i,j})$ are given by
 $$
 a_{i,j}=d(v'_{i-1},v_j),
 $$
for all $0\leq i\leq m-1, \; 0\leq j\leq n-1$.
 \end{thm}
  \begin{proof}
  The main idea of the proof is to include the $n$-gon $v$ and the $m$-gon $v'$ into
  a bigger $N$-gon in a Farey graph, and then apply Eq.~(\ref{Frizaij}).
  In other words, we will include the fundamental domain~$P$ into a (bigger) frieze pattern.
  
  First, let us show that
  $$
  v'_{m-2}>v_0>v_1>\ldots>v_{n-1}>v'_{m-1}.
  $$
Indeed, the vertices $v'_{m-2}, v'_{m-1},v'_0$ are consecutive vertices of the $m$-gon $v'$.
By assumption, $v'_{m-1}=\frac{0}{1}$, so that the condition
$$
d(v'_{m-2}, v'_{m-1})=1
$$
implies
$v'_{m-2}=\frac{1}{\ell}$ for some $\ell$.
By Lemma~\ref{DistLem}, the distance $d(v'_{0}, v'_{m-2})$ coincides with the number
of triangles at the vertex $v'_{m-1}$ which is, by construction, equal to $q'_0$.
We finally have:
$$
d(v'_{0}, v'_{m-2})=\ell=q'_0,
$$
  so that  $v'_{m-2}=\frac{1}{q'_0}$.
  The inequality  $v'_{m-1}>v_0$ then follows from the second inequality~(\ref{StrangeEq}).
  
  It is well-known that the Farey graph is connected; see~\cite{HaWr}.
  Therefore, two disjoint polygons, $v$ and $v'$, belong to some $N$-gon
   that contain the $n$-gon $v$ and the $m$-gon $v'$.
  
  Theorem~\ref{theprop} then follows from formula~(\ref{Frizaij}).
  \end{proof}
  
 \begin{ex}
 Consider the tiling given in Figure \ref{exTiling}. 
 It corresponds to the following data
 $$
 q=(1,2,2,1,3), \qquad q'=(2,1,2,1), \qquad 
 M=\begin{pmatrix}2&5\\[4pt]
 7&18\end{pmatrix}.
 $$
 The associated $5$-gon and $4$-gon in the Farey graph are
 as follows: 
 $$
 v=\left(\frac{2}{7},\,\frac{5}{18},\,\frac{8}{29},\,\frac{11}{40},\,\frac{3}{11}\right),
  \quad \hbox{and}\quad
  v'=\left(\frac{1}{0},\,\frac{1}{1},\,\frac{1}{2},\,\frac{0}{1}\right), 
$$
respectively.
They can be included in an $11$-gon; see Figure~\ref{HandO}.
   \begin{figure}[!h]
\begin{center}
\includegraphics[width=14cm]{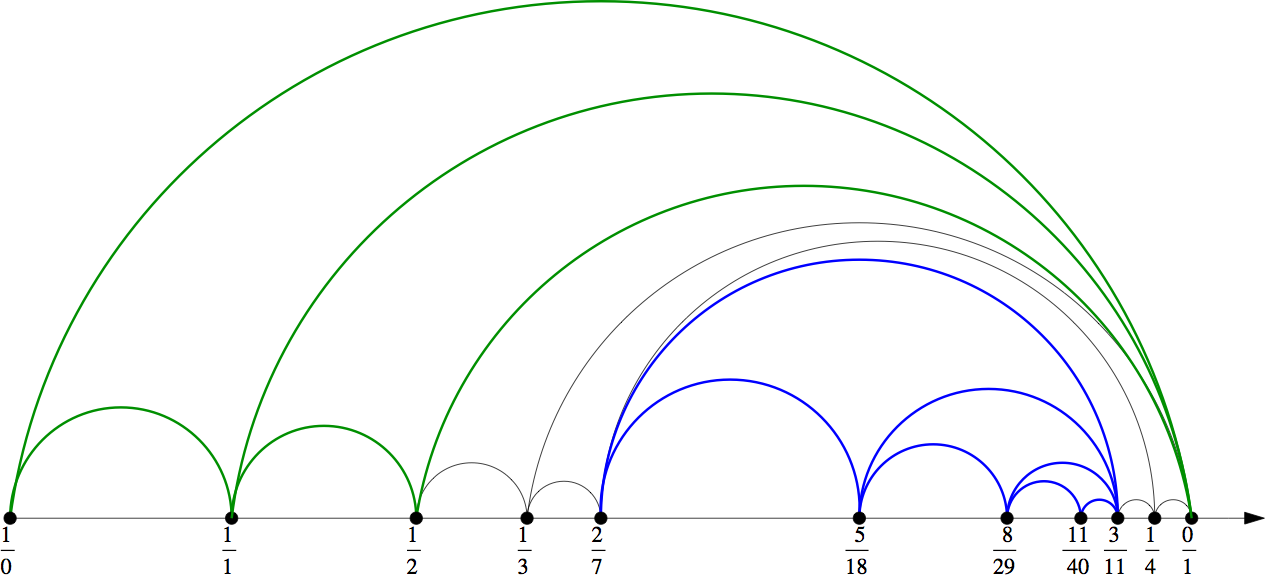}
\end{center}
\caption{The subgraph associated with the tiling in Figure \ref{exTiling}}
\label{HandO}
\end{figure}
\end{ex}

\bigskip
{\bf Acknowledgments}.  
We are grateful to Pierre de la Harpe for helpful comments.
S.~M-G. and V.~O. were partially supported by the PICS05974 ``PENTAFRIZ'' of CNRS;
S.T. was partially supported by the NSF grant DMS-1105442.


\end{document}